\documentclass[11pt, reqno]{amsart}

\usepackage{amsmath,amsfonts,amsthm,amssymb,color,hyperref, mathrsfs, graphicx, mathabx}

\usepackage[T1]{fontenc}

\usepackage[left=1.1in, right=1.1in, top=1.1in,bottom=1.1in]{geometry} 

\newtheorem{theorem}{Theorem}[section]
\newtheorem{lemma}[theorem]{Lemma}
\newtheorem{proposition}[theorem]{Proposition}

\theoremstyle{definition}

\theoremstyle{remark}
\newtheorem{remark}{Remark}
\numberwithin{equation}{section}

\newcommand{\R}{\mathbb R}

\newcommand{\E}{\mathbb E}
 \newcommand{\FF}{\mathscr{F}}
 \newcommand{\cF}{\mathcal{F}}
\newcommand{\fH}{\mathfrak{H}}
 \newcommand{\PP}{\mathbb P}
\newcommand{\B}{\mathbb B}
\newcommand{\bi}{\textbf{i}}

\newcommand{\1}{\textbf{1}}
\newcommand{\e}{{\varepsilon}}
 
 \newcommand{\wh}{\widehat}

\begin{document}

\title[Spatial ergodicity of stochastic wave equations in dimensions 1,2 and 3]{Spatial ergodicity of stochastic wave equations in dimensions 1,2 and 3}

 \date{\today}

\author[D. Nualart]{David Nualart$^{\ast,1}$}

\author[G. Zheng]{Guangqu Zheng$^{\ast,2}$}

\maketitle

\vspace{-0.5cm}

\begin{center}  {\small \textit{University of Kansas}$^\ast$; nualart@ku.edu$^1$, zhengguangqu@gmail.com$^2$ } \end{center}

\begin{abstract}  In this note, we study a large class of stochastic wave equations with spatial dimension less than or equal to $3$.  Via a soft application of Malliavin calculus, we establish that their random field solutions are spatially ergodic.   \end{abstract}


\allowdisplaybreaks

\section{Introduction}

    In this article, we fix $d\in\{1,2,3\}$ and
consider the  stochastic wave equation  
\begin{equation}  \label{SWE}
\dfrac{\partial^2 u}{\partial t^2} = \Delta u + \sigma(u)   \dot{W},     \\
\end{equation}
on $\R_+\times \R^d$ with initial conditions 
$ u(0,x)=1$ and $ \frac {\partial  u} {\partial t} (0,x)=0 $,  where $\Delta$ is Laplacian in  the space variables and $\dot{W}$ is a centered Gaussian  noise with
covariance
\begin{equation}
\E[ \dot{W}(t,x)\dot{W}(s,y)    ] =  \delta_0(t-s) \gamma( x-y).    \label{cov}
\end{equation}

\medskip 

\emph{Throughout this article}, we fix the following conditions:
\begin{itemize}
 \item[\textbf{(C1)}]  $\sigma:\R\to\R$ is Lipschitz continuous with Lipschitz constant $L\in(0,\infty)$.
 
 \item[\textbf{(C2)}] $\gamma$ is a tempered  nonnegative and nonnegative definite measure, whose Fourier transform  $\mu$  satisfies  Dalang's condition:
  \begin{align}\label{DalangC}
  \int_{\R^d} \frac{\mu(dz )}{1+  | z| ^2} < \infty,
 \end{align}
 where $ | \cdot | $ denotes the Euclidean norm on $\R^d$.
\end{itemize}
Conditions  \textbf{(C1)} and  \textbf{(C2)}   ensure that    equation \eqref{SWE} has a unique  \emph{random field solution}, which is adapted to the filtration generated by $W$, such that $\sup\big\{  \E \big[  \vert u(t,x)\vert^k\big]:   (t,x) \in [0,T]\times \R^d\big\}      $ is finite for all  $T\in(0,\infty)$ and  $k\ge 2$,  and  
\begin{equation}\label{mild}
 u(t,x) = 1+    \int_0^t \int_{\R^d}  G(t-s,x-y) \sigma(u(s,y))W(ds,dy), 
\end{equation}
where   the above stochastic integral is defined in the sense of Dalang-Walsh and $G(t-s,x-y)$ denotes the fundamental solution to the corresponding deterministic  wave equation, \emph{i.e.} 
\begin{align}\label{funS}
G(t,\bullet) :=\begin{cases}
  \dfrac{1}{2} \1_{\{ | \bullet| < t\}} ,   &\text{if $d=1$} \\
 \dfrac{1}{2\pi \sqrt{t^2 - | \bullet|^2}} \1_{\{| \bullet | <t \}},  &\text{if $d=2$}\\
 \dfrac{1}{4\pi t} \sigma_t, &\text{if $d=3$},
\end{cases}
\end{align}
with $\sigma_t$ denoting the surface measure on $\partial \B_t: = \{ x\in\R^3: |x| =t \}$; see Example 6 and Theorem 13 in Dalang's paper \cite{Dalang99}.   The proof of \cite[Theorem 13]{Dalang99} follows from a standard Picard iteration scheme, from which one can see that $u(t,x)\equiv1$ if $\sigma(1)=0$.

It is not difficult to see that for each fixed $t>0$, $\big\{ u(t,x): x\in\R^d\big\}$ is strictly stationary meaning its law is invariant under spatial shift. 
Indeed,  
for each $y\in \R^d$,  the random field $\{ u(t,x+y) : x\in \R^d\}$  coincides almost surely with  the random field $u$ driven by the shifted  noise $W_y$ given by
\[
W_y(\phi)= \int_{\R_+} \int_{\R^d} \phi(s,x-y) W(ds,dx), ~\phi\in C_c^\infty(R_+\times \R^d)
\]
  The noise $W_y$ has the same distribution as $W$, which is enough for us to conclude the stationarity property.  We refer readers to  Lemma 7.1 in \cite{CKNP19} and footnote 1 in \cite{DNZ18} for similar arguments. 

  Then it is natural to define    an associated family of shifts $\{\theta_y: y\in\R^d\}$ by setting
 \[
 \theta_y(\{u(t,x), x\in \R^d\})=  \{ u(t, x+y) ,x\in \R^d \},
 \]
 which  preserve the law of the process.   Then the following question arises:
    \begin{center}
\emph{  Are the   invariant sets for   $\{\theta_y:y\in\R^d\}$  trivial? }
    \end{center}
That is,   for each fixed $t>0$,     is $\big\{ u(t,x): x\in\R^d\big\}$ ergodic? See the book  \cite{Peterson90} for more account on ergodic theory.   In the following theorem, we provide an affirmative answer to the above question.

  \begin{theorem}\label{Mainthm}  Assume that the spectral measure has no atom at zero, \emph{i.e.} $\mu\big( \{ 0 \} \big) = 0$, then for each $t>0$, $\{ u(t,x): x\in\R^d\}$ is ergodic. 
\end{theorem}
    
Condition $\mu\big( \{ 0 \} \big) = 0$ echoes Maruyama's early work \cite{Maruyama49} on ergodicity of stationary Gaussian processes and it also finds its place in the recent work of Chen, Khoshnevisan, Nualart and Pu \cite{CKNP19} on the solution to stochastic heat equations.

\begin{remark}\label{REM1}
  Under Dalang's condition \eqref{DalangC}, property $\mu\big( \{ 0 \} \big) = 0$ is equivalent to 
$\gamma(\B_R)=o(R^d)$, as $R\rightarrow  +\infty$;  see \cite[Theorem 1.1]{CKNP19}. Here and throughout the paper we will make use of the notation
$\B_R= \{x\in \R^d: |x| \le R\}$ for any $R>0$. As a consequence, if $\gamma$ is a function, property $\mu\big( \{ 0 \} \big) = 0$ is equivalent to
\[
\lim_{R\rightarrow +\infty} \frac 1 {| \B_R|} \int_{\B_R} \gamma(x) dx =0,
\]
which means that the asymptotic average of $\gamma$ is zero.
\end{remark}

The ergodicity gives us the first-order result: With $\omega_d$ denoting the volume of $\B_1$,
 \[
\frac{1}{\omega_d R^d}\int_{\B_R} u(t,x)dx \xrightarrow{R\to\infty} 1
 \]
in $L^2(\Omega)$.
Then it is natural to investigate the corresponding second-order fluctuations.  They have been established in several cases briefly recalled below:
 
\begin{itemize}

\item When $d=1$, the Gaussian noise is white in time and behaves as a fractional noise in space with Hurst parameter $H\in[1/2,1)$, the authors of \cite{DNZ18} prove the Gaussian fluctuations for  spatial averages.

\item  The authors of \cite{BNZ20} investigate the case where   $d=2$ and $\gamma(z) = |z|^{-\beta}$ with $\beta\in(0,2)$.

\item In \cite{NZ20a}, we continued the study of the 2D stochastic wave equation when the covariance kernel $\gamma$ is integrable. 

\end{itemize}
Our Theorem \ref{Mainthm} (see also Remark \ref{REM1}) establishes the spatial ergodicity for all these cases. The key ingredient in the aforementioned references is a fundamental $L^p(\Omega)$-estimate of the Malliavin derivative of the solution:
\begin{align}\label{FEMD}
 \big\| D_{s,y} u(t,x) \big\|_p \lesssim G_{t-s}(x-y),
\end{align}
where $D$ is the Malliavin derivative operator defined over the isonormal Gaussian process $\big\{ W(\phi) : \phi\in\fH\big\}$ that will be defined in Section \ref{PRE}.  Such an inequality fails to work when $d=3$, as   the fundamental solution $G(t,\bullet)$ is a measure for $d=3$ (see  \eqref{funS}).  The Malliavin derivative $Du(t,x)$, unlike in previous works, is a random measure and it is not clear how to make sense of the left expression in \eqref{FEMD}. 
  We leave this problem for future research that will require some novel ideas in dealing with the Malliavin derivative.

\bigskip

The rest of the article is organized as follows: In Section \ref{PRE}, we briefly collect preliminary facts for our proofs that will be presented in Section \ref{proofs}.

\section{Preliminaries}\label{PRE}

  In this section we present some preliminaries on stochastic analysis and  Malliavin calculus.
    
    \subsection{Basic stochastic analysis}

Let  $\fH$ be defined  as the completion of $C_c(\R_+\times\R^d)$ under the inner product 
\[
\langle f, g\rangle_\fH = \int_{\R_+\times\R^{2d}} f(s,y) g(s, z) \gamma(y-z)dydzds =\int_{\R_+} \left(\int_{\R^d} \FF f(s,\xi) \FF g(s, -\xi) \mu(d\xi)\right)ds, 
\]
where $\FF f(s,\xi)=\int_{\R^d} e^{-\bi x\cdot \xi} f(s,x)dx$. Consider  an isonormal Gaussian process associated to the Hilbert space $\fH$, denoted by $W=\big\{ W(\phi): \phi\in \fH \big\}$. That is, $W$ is a centered  Gaussian family of random variables such that
$
\E\big[ W(\phi) W(\psi) \big] = \langle \phi, \psi\rangle_\fH$ for any   $\phi, \psi\in \fH$.
As the noise $W$ is white in time, a martingale structure naturally appears. First we define $\cF_t$ to be the $\sigma$-algebra generated by the $\PP$-negligible  sets and  the family of random variables $\big\{ W(\phi): \phi\in C^\infty(\R_+\times\R^d)$ has compact support contained in $[0,t]\times\R^d \big\}$, so we have a filtration $\mathbb{F}=\{\cF_t: t\in\R_+\}$. If $\big\{\Phi(s,y): (s,y)\in\R_+\times\R^d\big\}$ is  an $\mathbb{F}$-adapted  random field such that $\E\big[ \| \Phi\|_\fH^2 \big] <+\infty$, then 
\[
M_t = \int_{[0,t]\times\R^d} \Phi(s,y)W(ds,dy),
\]
interpreted as the Dalang-Walsh integral (\cite{Dalang99,NQS, Walsh}), is a square-integrable $\mathbb{F}$-martingale with quadratic variation  
\[
\langle M \rangle_t = \int_{[0,t]\times\R^{2d}} \Phi(s,y) \Phi(s,z) \gamma(y-z) dy dzds.
\]
A suitable    version of Burkholder-Davis-Gundy inequality (BDG for short) holds in this setting:  If $\big\{\Phi(s,y): (s,y)\in\R_+\times\R^d\big\}$ is an adapted random field with respect to $\mathbb{F}$ such that $ \| \Phi\|_\fH \in L^p(\Omega)$ for some $p\geq 2$, then
\begin{equation} \label{BDG}
\left\|  \int_{[0,t]\times\R^d} \Phi(s,y)W(ds,dy) \right\|_p^2 \leq 4p \left\| \int_{[0,t]\times\R^{2d}} \Phi(s,y)\Phi(s,z) \gamma(y-z) dydzds \right\|_{p/2};
\end{equation}
see \textit{e.g.} \cite[Theorem B.1]{Khoshnevisan}, where here $\| \bullet \|_p$ denotes the usual $L^p(\Omega)$-norm.

 \subsection{Malliavin calculus}
Now let us recall some basic facts on  the Malliavin calculus associated with $W$. For any unexplained notation and result,   we refer to the book \cite{Nualart}.  We denote by $C_p^{\infty}(\R^n)$ the space of smooth functions with all their partial derivatives having at most polynomial growth at infinity. Let $\mathcal{S}$ be the space of simple functionals of the form 
$F = f(W(h_1), \dots, W(h_n))
$ for $f\in C_p^{\infty}(\R^n)$ and $h_i \in \fH$, $1\leq i \leq n$. Then, the Malliavin derivative  $DF$ is the $\fH$-valued random variable given by
\begin{align*}
DF=\sum_{i=1}^n  \frac {\partial f} {\partial x_i} (W(h_1), \dots, W(h_n)) h_i\,.
\end{align*}
 The derivative operator $D$  is   closable   from $L^p(\Omega)$ into $L^p(\Omega;  \fH)$ for any $p \geq1$ and   we define $\mathbb{D}^{1,p}$ to be the completion of $\mathcal{S}$ under the norm
$$
\|F\|_{1,p} = \left(\E\big[ |F|^p \big] +   \E\big[  \|D F\|^p_\fH \big]   \right)^{1/p} \,.
$$
   The {\it chain rule} for $D$ asserts that if $F\in\mathbb{D}^{1,2}$ and $h:\R\to\R$ is Lipschitz, then $h(F)\in\mathbb{D}^{1,2}$  with 
\begin{equation} \label{chainrule}
D[h(F)] = h'(F) DF,
\end{equation}
where $h'$ denotes any version of the almost everywhere derivative (in view of Rademacher's theorem) satisfying 
\[
h(x) = h(0) + \int_0^x h'(t)dt \quad\text{for $x\geq 0$,} \quad h(0) = h(x) + \int_x^0 h'(t)dt\quad\text{for $x< 0$ }
\]
and $\|h' \|_\infty$ is bounded by the Lipschitz constant of $h$.  

We denote by $\delta$ the adjoint of  $D$ given by the duality formula 
\begin{equation} \label{D:delta}
\E[\delta(u) F] = \E[ \langle u, DF \rangle_\mathfrak{H}]
\end{equation}
for any $F \in \mathbb{D}^{1,2}$ and $u\in{\rm Dom} \, \delta \subset L^2(\Omega; \fH)$,  the domain of $\delta$. The operator $\delta$ is
also called the Skorohod integral and 
in our context, the Dalang-Walsh integral coincides with the Skorohod integral: Any   adapted random field $\Phi$ that satisfies $\E\big[ \|\Phi\|_\fH^2 \big]<\infty $ belongs to the domain of $\delta$  and
\begin{align}\label{EQ0}
\delta (\Phi) = 
\int_0^\infty \int_{\R^d} \Phi(s,y) W(d s, d y).
\end{align}
The operators $D,\delta$ satisfy the Heisenberg's commutation relation: \[
(D\delta - \delta D)(V) = V.
\]
 From this relation,   we have for any adapted random field $\Phi$ belonging to $\mathbb{D}^{1,2} (\fH)$  given as in \eqref{EQ0},
 \begin{equation}
 D_{s,y} \int_0^\infty \int_{\R^d} \Phi(r,z) W(d r, d z) = \Phi(s,y) + \int_0^\infty \int_{\R^d} D_{s,y}\Phi(r,z) W(d r, d z).
 \end{equation} \label{ECU1}

 It is known that for a random variable $F\in\mathbb{D}^{1,2}$, one can represent it as a stochastic integral:
 \[
 F = \E[F] + \int_{\R_+\times\R^d} \E\big[ D_{s,y}F \vert \mathcal{F}_s \big] W(ds,dy)
 \]
  (see \emph{e.g.} \cite[Proposition 6.3]{CKNP19}). This is known as  Clark-Ocone formula and it leads to the following Poincar\'e inequality:  For any such  two random variables $F,G\in\mathbb{D}^{1,2}$, we have
\begin{equation} \label{Poincare}
|{\rm Cov}(F,G )| \le  \int_0^\infty \int_{\R^{2d}}   \|  D_{s,y}F\|_{2}\|  D_{s,z}G\|_{2}  \gamma(y-z) dy dz ds.
\end{equation}

Throughout this note, we write $A\lesssim B$ to mean that $A\leq K B$ for some immaterial constant which may vary from line to line. 

\section{Proof of Theorem  \ref{Mainthm}}\label{proofs}

  We first introduce the following regularization of the kernel $G$:
Given a  nonnegative function $\psi\in C^\infty_c(\R^d)$ such that $\int_{\R^d} \psi(z)dz =1$, we define  $\psi_n(z) = n^d \psi(nz)$  for all $z\in \R^d$ and 
\begin{align}\label{def:GN}
G_n(t, x) = \int_{\R^d} G(t, dy) \psi_n(x-y).
\end{align} 
Here $G(t, dy)$ denotes $G(t,y)dy$, when $d=1,2$.
Consider the approximating sequence of random fields $ \{u_n\}_{ n\geq 1}$ defined by 
\begin{align}\label{APP}
u_n(t,x) = 1+ \int_0^t\int_{\R^d} G_n(t-s, x-y) \sigma \big(  u_n(s,y) \big) W(ds, dy).
\end{align}
It holds that,    for any $p\ge 1$
\begin{align}\label{uAPP}
\lim_{n\to+\infty}\sup_{(t,x)\in[0,T]\times\R^d} \big\| u_n(t,x) - u(t,x) \big\|_p =0
\end{align}
for any $T\in(0,\infty)$, see \cite[Proposition 1]{QSSS04}.   Fix $n\geq 1$ and   consider the Picard iteration scheme for $u_n$: We put $u_{n,0}(t,x)=1$ and for $k\geq0$,
\begin{align}
u_{n, k+1}(t,x) = 1 +  \int_0^t\int_{\R^d} G_n(t-s, x-y) \sigma\big( u_{n,k}(s,y) \big) W(ds,dy). \label{Picard}
\end{align}
 It is known that   for any $T>0$ and any $p\in [ 1,\infty)$,
\begin{align}\label{BDD:nk}
 \lim_{k\to+\infty}  \sup_{(t,x)\in[0,T]\times\R^d}\big\| u_{n,k}(t,x) - u_n(t,x) \big\|_p = 0;
\end{align}
the proof can be done following the same arguments as in the proof of   \cite[Theorem 13]{Dalang99}.  

\medskip

In the following, we present the key ingredient to prove our main result.
 
\begin{proposition} \label{prop:MD} Let $u_{n,k}$ be given as in \eqref{Picard} and fix $T\in(0,\infty)$. Then for any $p\geq 1$, the following estimate holds for all  $(t,x)\in [0,T] \times\R^d$ and for almost every $(s,y)\in [0,t]\times \R^d$
\[
\big\| D_{s,y} u_{n,k}(t,x) \big\|_p \lesssim \mathbf{1}_{\B_{\frac {a(k+1)} n+T}}(x-y),
\]
where $\B_a = \{ x\in\R^d: |x| \leq a\}$ contains the support of $\psi$ for some   $a>0$ and the implicit constant only depends on $(p, T, L, \gamma, n,k)$.   
\end{proposition}

Before we proceed with the proof of Proposition \ref{prop:MD} we 
show two technical lemmas.

\begin{lemma} \label{lem:DCT} Suppose   the Dalang's condition \eqref{DalangC} is satisfied. For any $T\in(0,\infty)$, we have
\begin{align}\label{DCT}
\mathcal{U}_T : = \sup_{b\in[0,T]}\int_{\R^d}  \big\vert \FF \mathbf{1}_{\B_b}(\xi)  \big\vert^2 \mu(d\xi) <\infty.
\end{align}
 
\end{lemma}

\begin{proof}

Let us recall from \cite[Lemma 2.1]{NZ19EJP} that 
${\displaystyle
\big\vert \FF \mathbf{1}_{\B_b}(\xi)  \big\vert^2 = \left\vert \int_{\B_b} e^{-\bi x \xi} dx \right\vert^2 = (2\pi b)^d | \xi |^{-d} J_{\frac d2}(b | \xi |)^2,
}$
 where, for $p>0$,
  \[
  J_p(x) : = \frac{(x/2)^p}{\sqrt{\pi} \Gamma(p+ \frac12)} \int_0^\pi (\sin \theta)^{2p} \cos(x \cos \theta) d\theta
 \]
 is the Bessel function of first kind with order $p$, which satisfies 
 \begin{itemize}
 \item[(i)] $\sup\big\{ | J_p(x) | : x\in\R_+ \big\} <\infty$,
 \item[(ii)] $| J_p(x) | \leq C | x|^{-1/2}$ for any $x\in\R$ and for some absolute constant $C>0$.
 \end{itemize}
 It is also clear that $| \FF \mathbf{1}_{\B_b} | \lesssim b^d$. Thus, 
  \begin{align*}
 \int_{\R^d} \big\vert \FF \mathbf{1}_{\B_b}(\xi)  \big\vert^2 \mu(d\xi) & =   \int_{| \xi| \leq 1} \big\vert \FF \mathbf{1}_{\B_b}(\xi)  \big\vert^2 \mu(d\xi)  +   \int_{| \xi| > 1} \big\vert \FF \mathbf{1}_{\B_b}(\xi)  \big\vert^2 \mu(d\xi)  \\
 &\lesssim  b^{2d}\mu\big( \{ \xi\in\R^d : | \xi | \leq 1\} \big)  + b^d\int_{| \xi | >1} | \xi |^{-d} J_{d/2}(b | \xi |)^2 \mu(d\xi) \\
 &\lesssim b^{2d} \mu\big( \{ \xi\in\R^d : | \xi | \leq 1\} \big) + b^{d-1} \int_{| \xi | >1} | \xi |^{-d-1}  \mu(d\xi)
  \end{align*}
  using point (ii) in the last step. Because of \eqref{DalangC} and $d\geq 1$, the two integrals in the last display are both finite. Hence the result \eqref{DCT} follows.
  \end{proof}
  
 \begin{lemma} \label{lem:GN}
  For each $n\geq 1$ and $T\in(0,\infty)$
  \[
 \Theta(T,n):= \sup_{(t,x)\in[0,T]\times\R^d} \big| G_n(t,x) \big| <\infty.
  \]
 \end{lemma} 
  
  \begin{proof} By definition,  
  \[
   \big| G_n(t,x) \big|  =  \int_{\R^d}  \psi_n(x- y)G(t,dy) \leq \| \psi_n \|_\infty G(t, \R^d).
  \]
  It is known that $\sup_{t\leq T}  G(t, \R^d)$ is finite  for any $T\in(0,\infty)$, so that $\Theta(T,n) <\infty$.
  \end{proof}
   
   \medskip
   
\begin{proof}[Proof of Proposition  \ref{prop:MD}]
 
 Recall the Picard iterations from \eqref{Picard}.
  Now let us fix $p\in[2,\infty)$, $T\in(0,\infty)$  and the integers $n,k$. Then, by standard arguments one can show that for any $(t,x)\in[0,T]\times\R^d$,
$u_{n,k+1}(t,x)$  belongs to the space $\mathbb{D}^{1,p}$ and, in view of \eqref{chainrule} and \eqref{ECU1},  we can write  for almost all
$(s,y)\in [0,t] \times \R^d$, 
\begin{align*}
D_{s,y}u_{n,k+1}(t,x) &= G_n(t-s,x-y) \sigma\big(u_{n,k}(s,y) \big)\\
& \qquad  + \int_s^t\int_{\R^d} G_n(t-r, x-z) \sigma'\big( u_{n,k}(r,z) \big) D_{s,y}u_{n,k}(r,z) W(dr,dz).
\end{align*}
Iterating this equation yields, with $r_0=t, z_0=x$,
  \begin{align}
   & D_{s,y} u_{n,k+1}(t,x)  = G_n(t-s,x-y)  \sigma\big( u_{n,k}(s,y) \big)\notag \\
   & \quad +  \int_s^t\int_{\R^d} G_n(t-r_1, x-z_1) \sigma'\big( u_{n,k}(r_1,z_1) \big) G_n(r_1 - s,z_1-y)   \sigma\big(  u_{n,k-1}(s,y) \big)  W(dr_1,dz_1) \notag \\
   &\qquad +   \sum_{\ell=2}^k   \sigma\big( u_{n,k-\ell}(s,y) \big) \int_s^t\cdots\int_s^{r_{\ell-1}} \int_{\R^{d\ell}}G_n(r_\ell -s, z_\ell-y)  \notag \\
   &\qquad\qquad  \times \prod_{j=1}^{\ell} G_n(r_{j-1} - r_{j}, z_{j-1} - z_{j})   \sigma'\big(  u_{n,k+1-j} (r_j,z_j) \big)  W(dr_j,dz_j)=:   \sum_{\ell=0}^k T_\ell.  \label{finiteit} 
  \end{align}   
  Note that by the uniform $L^p$-convergence of $u_{n,k}(t,x)$ as $k\rightarrow \infty$ and $n\rightarrow \infty$, we have
  \[
  \Lambda(T,p) :=  \sup_{n,k\geq 1}\sup_{t\in [0,T]} \sup_{x\in\R^d} \Big\| \sigma\big(u_{n,k}(t,x) \big) \Big\|_p <\infty.
  \]
 Now let us estimate $\| T_\ell \|_p$ for each $\ell\in\{0,1,\dots,  k\}$. 
  
  \bigskip
  
  \noindent\textbf{Case} $\ell=0$: It is clear that $\| T_0 \|_p \leq \Lambda(t,p) G_n(t-s,x-y)$. From now on, let us assume that 
  \begin{center} the support of $\psi$ is contained in $\B_a$ for some $a>0$. 
  \end{center}
Then  the function $x\in\R^d\longmapsto G_n(t,x)$ has a compact support that is contained  in $\B_{\frac a n+t}$.
So that   
\begin{align}\label{ineq:GN}
G_n(t-s,x-y)\leq  \Theta(t-s,n) \mathbf{1}_{\B_{\frac a n+t}}(x-y).
\end{align}
 It follows that
\begin{align}\label{0term}
\big\| T_0 \big\|_p \leq  \Lambda(T,p) \Theta(T,n) \mathbf{1}_{\B_{\frac a n +T}}(x-y).
\end{align}

    \bigskip
  
  \noindent\textbf{Case} $\ell=1$: By the BDG inequality \eqref{BDG},
  \begin{align*}
  \big\| T_1 \big\|_p^2 &\leq 4p \Bigg\|  \int_s^t dr_1\int_{\R^{2d}} dz_1dz'_1 G_n(t-r_1, x-z_1) \sigma'\big( u_{n,k}(r_1,z_1) \big) G_n(r_1 - s,z_1-y)  \\
   &   \qquad  \times G_n(t-r_1, x-z'_1) \sigma'\big( u_{n,k}(r_1,z'_1) \big) G_n(r_1 - s,z'_1-y)   \sigma^2\big(  u_{n,k-1}(s,y) \big) \gamma(z_1-z'_1)   \bigg\|_{p/2} \\
   &\leq 4pL^2 \Lambda(T,p)^2 \int_s^t dr_1\int_{\R^{2d}} dz_1dz'_1 G_n(t-r_1, x-z_1)  G_n(r_1 - s,z_1-y)  \\
   &   \qquad\qquad\qquad\qquad\qquad  \times G_n(t-r_1, x-z'_1)   G_n(r_1 - s,z'_1-y)    \gamma(z_1-z'_1).  
  \end{align*}
  Note that  a necessary condition for $G_n(t-r_1, x-z_1)  G_n(r_1 - s,z_1-y) \neq 0$  is
  \begin{center}
 $ x-z_1\in \B_{\frac an +t-r_1}$ and $z_1-y\in\B_{\frac an+r_1-s}$
  \end{center}
  which implies $x-y\in \B_{\frac {2a}n+t-s}$. This fact, together with Lemma \ref{lem:GN} and \eqref{ineq:GN}, leads to
  \begin{align*}
&  \big\| T_1 \big\|_p^2  \leq 4pL^2 \Lambda(T,p)^2  \Theta(T,n)^2 \mathbf{1}_{\B_{\frac {2a}n+T}}(x-y)  \\
   &\qquad \times \int_s^t dr_1\int_{\R^{2d}} dz_1dz'_1 G_n(t-r_1, x-z_1)     G_n(t-r_1, x-z'_1)     \gamma(z_1-z'_1) \\
   & \leq 4pL^2 (t-s) \Lambda(T,p)^2  \Theta(T,n)^4 \mathbf{1}_{\B_{\frac {2a}n+T}}(x-y) \int_{\R^{2d}} dz_1dz'_1  \mathbf{1}_{\B_{\frac a n+t}}(x-z_1)  \mathbf{1}_{\B_{\frac an+t}}(x-z'_1)     \gamma(z_1-z'_1) \\
   &\leq 4pL^2 (t-s) \Lambda(T,p)^2  \Theta(T,n)^4  \mathcal{U}_{\frac an+T}\mathbf{1}_{\B_{\frac{2a}n+T}}(x-y),
  \end{align*}
  by Lemma \ref{lem:DCT}. It follows that
  \begin{align}\label{1term}
   \big\| T_1 \big\|_p \leq 2  \sqrt{  p  \mathcal{U}_{\frac an+T} (t-s) } \Lambda(T,p)  \Theta(T,n)^2 \mathbf{1}_{\B_{\frac {2a}t+T}}(x-y).
  \end{align}
  
  \bigskip
  
    \noindent\textbf{Case} $\ell\in\{2, \dots, k\}$: We can first represent $T_\ell$ as
    \[
    T_\ell =    \int_s^t \int_{\R^d}    G_n(t-r_1, x-z_1) \sigma'\big( u_{n,k}(r_1,z_1) \big) \mathcal{J}(r_1,z_1) W(dr_1,dz_1),
    \]
    with $  \mathcal{J}(r_1,z_1) $ defined by 
    \begin{align*}
   \mathcal{J}(r_1,z_1 )&= \int_s^{r_1}\cdots\int_s^{r_{\ell-1}} \int_{\R^{d\ell-d}} \sigma\big( u_{n,k-\ell}(s,y) \big) G_n(r_\ell-s,z_\ell-y)\\
   &\qquad \times  \prod_{j=2}^{\ell} G_n(r_{j-1} - r_{j}, z_{j-1} - z_{j})   \sigma'\big(  u_{n,k+1-j} (r_j,z_j) \big)  W(dr_j,dz_j).
    \end{align*}
    In this way, we have 
    \begin{align*}
    \big\| T_\ell \big\|_p^2 & \leq 4p L^2  \int_s^t dr_1 \int_{\R^{2d}} dz_1dz_1'  \gamma(z_1-z_1')   G_n(t-r_1, x-z_1)G_n(t-r_1, x-z'_1) \| \mathcal{J}(r_1,z'_1)     \mathcal{J}(r_1,z_1)\|_{\frac{p}{2}}  \\
    &\leq 4pL^2    \int_s^t dr_1 \int_{\R^{2d}} dz_1dz_1'  \gamma(z_1-z_1')   G_n(t-r_1, x-z_1)G_n(t-r_1, x-z'_1) \|    \mathcal{J}(r_1,z_1)\|_{p} ^2,
    \end{align*}
    using symmetry and the fact that $\| XY\| _{p/2} \leq \| X\|_p \| Y\|_p \leq \frac{ \|X\|_p^2 + \| Y\|_p^2}{2}$. Iterating the above procedure for finite times yields
       \begin{align*}
    \big\| T_\ell \big\|_p^2 & \leq  \big( 4pL^2  \big)^{\ell-1}    \int_s^t  dr_1 \cdots \int_s^{r_{\ell-2}} dr_{\ell-1} \int_{\R^{2d\ell-2d}}   \|    \wh{\mathcal{J}}(r_{\ell-1}, z_{\ell-1})\|_{p}^2  \\
     &\qquad \qquad \times \prod_{j=1}^{\ell-1}  \gamma(z_j-z_j')   G_n(r_{j-1}-r_j, z_{j-1}-z_j)G_n(r_{j-1}-r_j, z_{j-1}-z'_j)  dz_jdz_j', 
    \end{align*}
    with $  \wh{\mathcal{J}}(r_{\ell-1}, z_{\ell-1})$ given by 
    \[
    \int_s^{r_{\ell-1}} \int_{\R^d}   \sigma\big( u_{n,k-\ell}(s,y) \big)  G_n(r_\ell-s, z_\ell-y) G_n(r_{\ell-1} - r_\ell, z_{\ell-1} - z_\ell ) \sigma'\big( u_{n,k+1-\ell}(r_\ell,z_\ell) \big) W(dr_\ell, dz_\ell).
    \]
   Similarly to how we estimate $\| T_1\|_p$, we get
   \[
   \big\|   \wh{\mathcal{J}}(r_{\ell-1}, z_{\ell-1}) \big\|_p^2 \leq 4pL^2 T \Lambda(T,p)^2 \Theta(T,n)^4 \mathcal{U}_{\frac an+T} \mathbf{1}_{\B_{\frac {2a}n+r_{\ell-1}- s}}(z_{\ell-1}-y).
   \]
   As in  Case $\ell=1$, we have the following implication:
   \[
   \mathbf{1}_{\B_{\frac {2a}n+r_{\ell-1}- s}}(z_{\ell-1}-y) \prod_{j=1}^{\ell-1}   G_n(r_{j-1}-r_j, z_{j-1}-z_j)G_n(r_{j-1}-r_j, z_{j-1}-z'_j)  \neq 0 \Longrightarrow x-y\in\B_{\frac {a(\ell+1)}n+t-s}.
   \]
   Note also that using \eqref{ineq:GN} and  integrating out $dz_{\ell-1} dz'_{\ell-1}$, \dots,  $dz_{2} dz'_{2}$ and $dz_1dz_1'$ yields
\begin{align*}
 & \quad \int_{\R^{2d\ell-2d}}    \prod_{j=1}^{\ell-1}  \gamma(z_j-z_j')   G_n(r_{j-1}-r_j, z_{j-1}-z_j)G_n(r_{j-1}-r_j, z_{j-1}-z'_j)  dz_jdz_j' \\
 &\leq \Theta(T,n)^{2\ell-2}   \int_{\R^{2d\ell-2d}}    \prod_{j=1}^{\ell-1}  \gamma(z_j-z_j')  \mathbf{1}_{\B_{\frac an+T}}(z_{j-1}-z_j)  \mathbf{1}_{\B_{\frac an+T}}(z_{j-1}-z'_j)  dz_jdz_j'\\
 & = \Theta(T,n)^{2\ell-2}  \left(  \int_{\R^{2d}}   \gamma(z-z')  \mathbf{1}_{\B_{\frac an+T}}(z)  \mathbf{1}_{\B_{\frac an+T}}(z')  dzdz'    \right)^{\ell-1} \leq  \Theta(T,n)^{2\ell-2}  \mathcal{U}_{\frac an+T}^{\ell-1},
     \end{align*}   
   where $\mathcal{U}_{\frac an+T}$ is defined in Lemma \ref{lem:DCT}. This leads to 
   \[
   \big\| T_\ell\big\|_p^2 \leq \frac{  \Lambda(T,p)^2 \Theta(T,n)^2 }{(\ell-1)!} \big( 4pL^2T \mathcal{U}_{\frac an+T}   \Theta(T,n)^2 \big)^\ell  \mathbf{1}_ { \B_{\frac {a(\ell+1)}n+T} }(x-y).
   \]
   Combining the above cases, we obtain ${\displaystyle 
   \big\| D_{s,y} v_{k+1}(t,x) \big\|_p \leq \sum_{\ell=0}^k \| T_\ell\|_p \lesssim    \mathbf{1}_ { \B_{\frac {a(k+1)}n+T} }(x-y).
   }$
   That is, Proposition \ref{prop:MD} is proved.    
   \end{proof}

We finally proceed with the proof of Theorem \ref{Mainthm}.

\begin{proof}[Proof of Theorem  \ref{Mainthm}] In view of \cite[Lemma 7.2]{CKNP19}, it suffices to prove
\[
V(R) := \text{Var} \left( R^{-d} \int_{\B_R} \prod_{j=1}^m g_j \big(  u(t,x+ \zeta^j) \big)    dx\right) \xrightarrow{R\to\infty} 0,
\]
for any fixed $\zeta^1, \dots,\zeta^m \in \R^d$ and $g_1, \dots, g_m\in C_b(\R)$ such that each $g_j$ vanishes at zero and has Lipschitz constant bounded by $1$. 

\medskip

  Using the elementary fact  $\text{Var}(X+Y) \leq 2 \text{Var}(X) + 2 \text{Var}(Y) $ for any two square-integrable random variables $X$ and $Y$, we write 
\begin{align*}
V(R) &\leq  2  \text{Var} \left( R^{-d} \int_{\B_R}  \mathcal{R}_{n,k}(x)   dx\right)   + 4  \text{Var} \left( R^{-d} \int_{\B_R}\Big[  \mathcal{R}_n(x)  -  \mathcal{R}_{n,k}(x) \Big]  dx \right)  \\
&\qquad\qquad +4  \text{Var} \left( R^{-d} \int_{\B_R}\Big[  \mathcal{R}(x)  -  \mathcal{R}_n(x) \Big]  dx \right):= 2 V_{1,n,k}(R) + 4 V_{2,n,k}(R) + 4 V_{3,n}(R),
\end{align*}
where
\[
\mathcal{R}(x) :=  \prod_{j=1}^m g_j \big(  u(t,x+ \zeta^j) \big), ~ \mathcal{R}_n(x) : = \prod_{j=1}^m g_j \big(  u_n(t,x+ \zeta^j) \big) ~ \text{and} ~ \mathcal{R}_{n,k}(x) : = \prod_{j=1}^m g_j \big(  u_{n,k}(t,x+ \zeta^j) \big).
\]
Using the stationarity and Minkowski's inequality,
\begin{align*}
 V_{3,n}(R)  & \leq \left\|  R^{-d} \int_{\B_R}\big[  \mathcal{R}_n(x)  -  \mathcal{R}(x) \big]  dx \right\|_2^2 \leq  \left(  R^{-d} \int_{\B_R}\big\|  \mathcal{R}_n(x)  -  \mathcal{R}(x) \big\|_2  dx \right)^2 \\
&=\omega_d^2 \Big\|  \mathcal{R}_n(0)  -  \mathcal{R}(0) \Big\|_2^2 \xrightarrow{n\to+\infty} 0,\quad \text{by \eqref{uAPP}}.
\end{align*}
The above limit takes place uniformly in $R>0$. Therefore, for any given $\e>0$, we can find $n \geq N_\e$ big enough such that $ V_{3,n}(R)  \leq \e$,  $\forall R>0$. \emph{From now on, let us fix such an integer $n$.}

\medskip
 
 Now let us estimate $V_{2,n,k}(R)$ similarly: Using Minkowski's inequality,

\begin{align*}
V_{2,n,k}(R) \leq \left\| R^{-d} \int_{\B_R}\Big[  \mathcal{R}_n(x)  -  \mathcal{R}_{n,k}(x) \Big]  dx \right\|_2^2 \leq \omega_d^2  \sup_{x\in\R^d} \Big\| \mathcal{R}_n(x)  -  \mathcal{R}_{n,k}(x) \Big\|_2^2 \xrightarrow{k\to\infty} 0,
\end{align*} 
 as a consequence of \eqref{BDD:nk}. So we can find some big $k\geq K_{\e, n}$ such that $V_{2,n,k}(R) \leq  \e$, $\forall R>0$. \emph{From now on, let us fix such an integer $k$.} 

\medskip

Finally, let us estimate the term $V_{1,n,k}(R)$: First  by using the Poincar\'e inequality  \eqref{Poincare}, we obtain
\begin{align}  \nonumber
V_{1,n,k}(R)& \leq  R^{-2d} \int_{\B_R^2} dx dy   \big\vert  \text{Cov} ( \mathcal{R}_{n,k}(x),  \mathcal{R}_{n,k}(y)) \big\vert  \\  \label{EQ2}
& \leq R^{-2d} \int_{\B_R^2} \int_0^t  \int_{\R^{2d}}    \| D_{s,z} \mathcal{R}_{n,k}(x) \| _2\| D_{s,z'} \mathcal{R}_{n,k}(y) \| _2 \,\gamma(z-z')dz dz' ds  dxdy .
\end{align} 
By the chain rule \eqref{chainrule},
\[
\big\vert D_{s,z} \mathcal{R}_{n,k}(x) \big\vert \leq \mathbf{1}_{(0,t)}(s)  \sum_{j_0=1}^m \left\vert \prod_{j=1, j\not = j_0} ^m g_j(u_{n,k}(t,x+ \zeta^j)) \right\vert   \big\vert D_{s,y} u_{n,k}(t,x+ \zeta^{j_0}) \big\vert,
\]
which implies, for any $s\in [0,t]$, 
\begin{equation} \label{EQ1}
\| D_{s,z} \mathcal{R}_{n,k}(x) \|_2   \le  \max_{1\le j \le m} \sup_{a\in \R} | g_j(a) | ^{m-1} \sum_{j_0=1}^m       \big\|  D_{s,y} u_{n,k}(t,x+ \zeta^{j_0}) \big\|_2  \lesssim  \sum_{j_0=1}^m \1_{\B_b}\big(x-y +\zeta^{j_0}\big),
\end{equation}
where  $b= a(k+1)/n+T$, as a consequence of Proposition \ref{prop:MD}. Plugging  \eqref{EQ1} into \eqref{EQ2}, yields
\begin{align*}  
V_{1,n,k}(R)& \lesssim  R^{-2d}\sum_{j,\ell =1}^m   \int_{\B_R^2} \int_0^t  \int_{\R^{2d}}   \1_{\B_b}(x-z +\zeta^{j}) \1_{\B_b} (y-z' +\zeta^{\ell})\gamma(z-z')dz dz' ds  dxdy \\
&\lesssim  R^{-2d}\sum_{j,\ell =1}^m   \int_{\B_R^2}   \int_{\R^{2d}}   \1_{\B_b}(x-z +\zeta^{j}) \1_{\B_b} (y-z' +\zeta^{\ell})\gamma(z-z')dz dz'   dxdy. 
\end{align*} 
Therefore using Fourier transform, we write
\begin{align*}
 \wh{V} :=& \int_{\B_R^2}   \int_{\R^{2d}}   \1_{\B_b}(x-z +\zeta^{j}) \1_{\B_b} (y-z' +\zeta^{\ell})\gamma(z-z')dz dz'   dxdy\\
 & =  \int_{\B_R^2}   \int_{\R^{d}}   e^{-\bi (x-y+\zeta^j -\zeta^\ell )\xi} \big\vert  \FF \1_{\B_b}( \xi)  \big\vert ^2 \mu(d\xi)dxdy.
 \end{align*}
 Put $\ell_R(\xi) = \int_{\B_R^2}  e^{-\bi (x-y  )\xi} dxdy   $, which is a nonnegative function. So we get 
 \begin{align*}
 \wh{V} &\leq    \int_{\R^{d}}  \ell_R(\xi) \big\vert  \FF \mathbf{1}_{\B_b}( \xi)  \big\vert ^2 \mu(d\xi) dxdy =  \int_{\B_R^2}  \int_{\R^{d}}   e^{-\bi (x-y  )\xi}  \big\vert  \FF \mathbf{1}_{\B_b}( \xi)  \big\vert ^2 \mu(d\xi) dxdy \\
 &= R^{2d} \int_{\B_1^2}  \int_{\R^{d}}   e^{-\bi R(x-y  )\xi}  \big\vert  \FF \mathbf{1}_{\B_b}( \xi)  \big\vert ^2 \mu(d\xi) dxdy.
 \end{align*}
 That is, 
 \begin{align*}
 R^{-2d}\wh{V} &\leq \int_{\R^{d}}  \left(\int_{\B_1^2}    e^{-\bi R(x-y  )\xi}  dx dy  \right)   \big\vert  \FF \mathbf{1}_{\B_b}( \xi)  \big\vert ^2 \mu(d\xi) 
  \end{align*}
Since $\mu(\{ 0\})=0$, for $\mu$-almost every $\xi$,  $\int_{\B_1^2}    e^{-\bi R(x-y  )\xi}  dx dy$ converges to zero as $R\to\infty$, by Riemann-Lebesgue's lemma.    Thus, by dominated convergence theorem with the dominance condition \eqref{DCT}, we deduce that
$ R^{-2d}\wh{V} $ converges to zero as $R\to+\infty$. This leads to 
$
 V_{1,n,k}(R) \to 0
$, as $R\to+\infty$.
It follows that
$
\limsup_{R\to+\infty} V(R) \leq  8\e,
$
where $\e>0$ is arbitrary. Hence we can conclude our proof. 
\end{proof}



 \noindent\textbf{Acknowledgment:}  D. Nualart is supported by NSF Grant DMS 1811181.

 \end{document}